\documentclass[amssymb,reqno,amsfonts,refcheck,12pt,verbatim,righttag]{amsart}

\usepackage{graphicx}
\usepackage{graphics,color}
\usepackage{amssymb,mathrsfs}
\usepackage{graphicx,color,tikz,caption,subcaption}
\usepackage[all]{xy}
\usepackage{mathtools}
\usepackage{enumerate}
\usepackage{amsmath}
\usepackage{bbm}

\usepackage[colorlinks, linkcolor=blue,citecolor=blue, anchorcolor=blue,urlcolor=blue,pagebackref,hypertexnames=false]{hyperref}

\numberwithin{equation}{section} 
\newtheorem{thm}{Theorem}[section]
\newtheorem{lem}[thm]{Lemma}
\newtheorem{prop}[thm]{Proposition}

\newtheorem{rem}[thm]{Remark}

\newcommand{\R}{\mathbb{R}}

\newcommand{\N}{\mathbb{N}}

\usepackage{float}

\begin{document}
\baselineskip 14pt

\title{Maximal run-length function with constraints: a generalization of the Erd{\H o}s-R{\'e}nyi limit theorem and the exceptional sets}

\author{Yu-Feng Wu}
\address[]{School of Mathematics and Statistics\\ Central South University\\ Changsha, 410085, PR China}
\email{yufengwu.wu@gmail.com}



\date{}

\begin{abstract}
Let $\mathbf{A}=\{A_i\}_{i=1}^{\infty}$ be a sequence of sets with each $A_i$ being a non-empty collection of $0$-$1$ sequences  of  length $i$. For $x\in [0,1)$, the maximal run-length function $\ell_n(x,\mathbf{A})$ (with respect to $\mathbf{A}$) is defined to the largest $k$
such that in  the first $n$ digits of the dyadic expansion of $x$
there is a consecutive subsequence contained in $A_k$. Suppose  that $\lim_{n\to\infty}(\log_2|A_n|)/n=\tau$ for some $\tau\in [0,1]$ and one additional assumption holds, we prove a generalization of the Erd{\H o}s-R{\'e}nyi limit theorem which states that
\[\lim_{n\to\infty}\frac{\ell_n(x,\mathbf{A})}{\log_2n}=\frac{1}{1-\tau}\]
 for Lebesgue almost all $x\in [0,1)$. For the exceptional sets,  we prove under a certain stronger assumption on $\mathbf{A}$ that the set 
\[\left\{x\in [0,1): \lim_{n\to\infty}\frac{\ell_n(x,\mathbf{A})}{\log_2n}=0\text{ and } \lim_{n\to\infty}\ell_n(x,\mathbf{A})=\infty\right\}\]
has Hausdorff dimension at least $1-\tau$. 
\end{abstract}

\maketitle

\section{Introduction}\label{S-1}

We first recall the limit theorem of Erd{\H o}s and R{\'e}nyi on the maximal run-length function for the dyadic expansions of real numbers. It is well-known that every $x\in [0,1)$ can be represented as an infinite series
\begin{equation}\label{eqxdya}
x=\sum_{k=1}^{\infty}\frac{x_k}{2^k},
\end{equation}
where $\{x_k\}_{k=1}^{\infty}$ is an infinite $0$-$1$ sequence defined by  $x_k=\lfloor 2\{2^{k-1}x\}\rfloor$ for all $k\geq 1$. Throughout this paper, $\lfloor \cdot\rfloor$ and $\{\cdot\}$ stand for the integral and fractional parts of a real number, respectively. We call \eqref{eqxdya} the {\it dyadic expansion} of $x$ and write it for short as $x=0.x_1x_2\ldots$, and call $x_k$ ($k\geq 1$) the {\it digits} of $x$. 

In 1970, Erd{\H o}s and R{\'e}nyi \cite{ErdosRenyi} proved a celebrated law of large numbers which yields the asymptotic behavior of the length of the longest consecutive $1$'s in the first $n$ digits of  almost all $x$. More precisely, for $n\in\N$ and $x\in [0,1)$ with dyadic expansion $x=0.x_1x_2\ldots$, the {\it maximal run-length function} $r_n(x)$ is defined as 
\begin{equation*}
r_n(x)=\max\{k: x_{i+1}=\cdots=x_{i+k}=1\text{ for some }0\leq i\leq n-k\}.
\end{equation*}
(We always define $\max\emptyset$ as $0$.)
Erd{\H o}s and R{\'e}nyi \cite{ErdosRenyi} proved that for Lebesgue almost all $x\in [0,1)$, 
\begin{equation}\label{eqrnxl}
\lim_{n\to\infty}\frac{r_n(x)}{\log_2n}=1.
\end{equation}
That is, the rate of growth of $r_n(x)$  is $\log_2n$ for almost all $x$ in $[0,1)$. 

From the above result of Erd{\H o}s and R{\'e}nyi we know that  the numbers in $[0,1)$ that violate  \eqref{eqrnxl} form a Lebesgue null set, it is thus natural to study the corresponding exceptional sets from the viewpoint of dimension theory. Ma, Wen and Wen \cite{MaWenWen} first studied this question and   proved that the set of points that do not satisfy \eqref{eqrnxl} has full Hausdorff dimension. Indeed, they proved that the set 
 \[\left\{x\in [0,1): \lim_{n\to\infty}\frac{r_n(x)}{\log_2n}=0\text{ and }\lim_{n\to\infty}r_n(x)=\infty\right\}\]
 has Hausdorff dimension one. Later, Zou \cite{Zou11} investigated the set of $x\in [0,1)$ such that $r_n(x)$ obeys other asymptotic behavior than $\log_2n$. He proved that the set 
 \[\left\{x\in [0,1): \lim_{n\to\infty}\frac{r_n(x)}{\delta_n}=1\right\}\]
also has Hausdorff dimension one whenever $\{\delta_n\}_{n=1}^{\infty}$ is a nondecreasing sequence of positive integers satisfying  $\lim_{n\to\infty}\delta_n=\infty$ and $\lim_{n\to\infty}\delta_{n+\delta_n}/\delta_n=1$.  There have been other  works on various kinds of exceptional sets and level sets related to the Erd{\H o}s-R{\'e}nyi limit theorem, see \cite{LiWu16,LiWuon17,SUXU17}. For the definition and basic properties of Hausdorff dimension, the reader is referred to \cite{Falconer14}.   

There have been several generalizations of the Erd{\H o}s-R{\'e}nyi limit theorem. Let $\beta>1$ be a real number. Tong {\it et al.} \cite{TongYuZhao16} proved that the maximal length of consecutive zero digits in the $\beta$-expansions of numbers in $[0,1)$ obeys an asymptotic behavior similar to \eqref{eqrnxl} (with $\log_2n$ being replaced by $\log_{\beta}n$). They also showed that the corresponding level sets all have Hausdorff dimension one. For studies on other related exceptional sets in this setting, see \cite{GaoHuLi18A,LiuLu17H,LiuLuZhang18O,LXZheng19}. Chen and Yu \cite{ChenYu18A} generalized the Erd{\H o}s-R{\'e}nyi limit theorem from another perspective. More precisely, let $m\geq 2$ and $p<m$ be two positive integers, $A\subset \{0,1,\ldots,m-1\}$ be a subset of $p$ elements. For $n\in\N$ and  $x\in [0,1)$ with $m$-ary expansion $x=\sum_{i=1}^{\infty}\frac{x_i}{m^i}$, define 
\begin{equation}\label{eqRnpx}
R_n^p(x)=\max\{k: x_{i+1},\ldots, x_{i+k}\in A\text{ for some }0\leq i\leq n-k\}.
\end{equation}
Chen and Yu \cite{ChenYu18A} proved that for Lebesgue almost all $x\in [0,1)$, 
\begin{equation}\label{eqCY}
\lim_{n\to\infty}\frac{R_n^p(x)}{\log_{m/p}n}=1.
\end{equation}
They also showed that the collection of $x\in [0,1)$ such that \eqref{eqCY} is not satisfied has full Hausdorff dimension. Very recently, L{\"u} and Wu \cite{LFWJ20} obtained another form of generalization of the Erd{\H o}s-R{\'e}nyi limit theorem. Although L{\"u} and Wu proved their limit theorem in the setting of $\beta$-expansions for all $\beta>1$, below I choose to only state it for the special case that $\beta=2$, since it is more closely related to this paper.  Fix $y\in [0,1)$ and let $0.y_1y_2\ldots$ be its dyadic expansion. For $n\in\N$ and  $x\in [0,1)$ with dyadic expansion $0.x_1x_2\ldots$, define 
\[r_n(x;y)=\max\{k: x_{i+1}=y_1,\ldots,x_{i+k}=y_k\text{ for some }0\leq i\leq n-k\}.\]
That is, $r_n(x;y)$ is the maximal length of the prefix of the dyadic expansion of $y$ that appears in the first $n$ digits of $x$. Then a special case of \cite[Theorem 1]{LFWJ20} states that 
\begin{equation}\label{eqLuWu}
\lim_{n\to\infty}\frac{r_n(x;y)}{\log_2n}=1. 
\end{equation}
for Lebesgue almost all $x\in [0,1)$.

The purpose of this paper is to introduce a general framework which can be understood as  {\it the maximal run-length function with constraints}, and prove a limit theorem which simultaneously generalizes the limit theorems of Erd{\H o}s and R{\'e}nyi (\eqref{eqrnxl}), Chen and Yu (\eqref{eqCY}), and \eqref{eqLuWu} due to L{\"u} and Wu. We will also discuss the corresponding exceptional sets. 

To introduce our setting in this paper, we first introduce some notation. For $k\in\N$, let $\Sigma_k$ be the set of $0$-$1$ sequences of length $k$, i.e., 
\[\Sigma_k=\{x_1\ldots x_k: x_i\in\{0,1\}, 1\leq i\leq k\}.\]
We call each element of $\Sigma_k$ a {\it word}. For $\ell\geq k$, a word $u\in\Sigma_k$ is called a {\it subword} of a word $v\in\Sigma_{\ell}$ if $u$ appears as a consecutive subsequence of $v$. For $w_1\in\Sigma_{k_1}$ and $w_2\in\Sigma_{k_2}$, let $w_1w_2\in\Sigma_{k_1+k_2}$ be the concatenation of $w_1$ and $w_2$.   

Let $\mathbf{A}=\{A_k\}_{k=1}^{\infty}$ be a sequence of sets where each  $A_k$ is a non-empty subset of $\Sigma_k$. For $n\in\N$ and $x\in [0,1)$ with dyadic expansion $x=0.x_1x_2\ldots$, we define the {\it maximal run-length function} (with respect to $\mathbf{A}$), denoted by $\ell_n(x,\mathbf{A})$, as 
\[\ell_n(x,\mathbf{A})=\max\{k: x_{i+1}\ldots x_{i+k}\in A_k\text{ for some }0\leq i\leq n-k\}.\] 
Concerning the asymptotic behavior of $\ell_n(x,\mathbf{A})$, we have the following result, which involves the growth rate of the cardinalities of $A_k$ ($k\geq1$). 

Let  $|F|$ denote the cardinality of a finite set $F$.  We conventionally define $\frac{1}{0}:=\infty$.   

\begin{thm}\label{thm1}
Let $\mathbf{A}=\{A_k\}_{k=1}^{\infty}$ be a sequence of sets where  each  $A_k$ is a non-empty subset of $\Sigma_k$. Suppose that the following hold:
\begin{itemize}
\item[(i)]For any $k\geq 1$, $w_1\ldots w_{k+1}\in A_{k+1}$ implies that $w_1\ldots w_k\in A_k$;
\item[(ii)]$\lim_{k\to\infty}(\log_2|A_k|)/k=\tau$ for some $\tau\in [0,1]$.
\end{itemize}
Then we have 
\begin{equation}\label{eqntoinell}
\lim_{n\to\infty}\frac{\ell_n(x,\mathbf{A})}{\log_2n}=\frac{1}{1-\tau}
\end{equation}
for Lebesgue almost all $x\in [0,1)$. 
\end{thm}

We remark that we can derive  from Theorem \ref{thm1} the limit theorems of Erd{\H o}s and R{\'e}nyi (\eqref{eqrnxl}), Chen and Yu (\eqref{eqCY}), and \eqref{eqLuWu} due to L{\"u} and Wu. Indeed, in Theorem \ref{thm1}, by taking $A_k=\{1^k\}$ for $k\geq 1$ (here $1^k$ stands for the word consists of $k$ many $1$'s), we immediately obtain the limit theorem of Erd{\H o}s and R{\'e}nyi.  Given a fixed  $y\in [0,1)$ with dyadic expansion $0.y_1y_2\ldots$, for $k\geq1$ let $A_k$ be the singleton $\{y_1\ldots y_k\}$, then we recover from Theorem \ref{thm1} the result \eqref{eqLuWu} of L{\"u} and Wu. As for the limit theorem of Chen and Yu, we first remark that Theorem \ref{thm1} can be directly extended to the case for  $m$-ary expansions ($m\geq2$ is an integer). Given a non-empty set $A\subset \{0,1,\ldots,m-1\}$  with $p<m$ elements, for $k\geq1$ we take $$A_k=A^k:=\{x_1\ldots x_k: x_i\in A,1\leq i\leq k\}.$$  
Then notice that $\ell_n(x,\mathbf{A})$  coincides with $R_n^p(x)$ defined in \eqref{eqRnpx}. Hence a straightforward extension of Theorem \ref{thm1} gives that for Lebesgue almost all $x\in [0,1)$,
\[\lim_{n\to\infty}\frac{R_n^p(x)}{\log_mn}=\frac{1}{1-\tau},\quad \text{with }\tau=\lim_{k\to\infty}\frac{\log_m|A_k|}{k}=\log_mp.\]
This is easily seen to be equivalent to \eqref{eqCY}. Thus we derive the result  \eqref{eqCY} of Chen and Yu from (an extension of) Theorem \ref{thm1}.

Regarding  the set of points $x\in [0,1)$ that do not satisfy \eqref{eqntoinell}, by replacing the assumption (i) in Theorem \ref{thm1} by a stronger one, we prove that the set has Hausdorff dimension at least $1-\tau$. Indeed, we have the following result, in which $\log_2n$ is replaced by a more general sequence $\phi(n)$ satisfying a certain growth condition.  

\begin{thm}\label{thm2}
Let $\mathbf{A}=\{A_k\}_{k=1}^{\infty}$ be a sequence of sets where  each  $A_k$ is a non-empty subset of $\Sigma_k$, $\{\phi(k)\}_{k=1}^{\infty}$ be a nondecreasing sequence of positive numbers satisfying $\lim_{k\to\infty}\phi(k)=\infty$ and  $\lim_{k\to\infty}\frac{\phi(k)}{\phi(1)+\cdots+\phi(k-1)}=0$. Let $\tau=\limsup_{k\to\infty}(\log_2|A_k|)/k$. Suppose that the following holds: for any $1\leq i<j$,
\begin{equation}\label{assumption}
\text{all subwords of length } i \text{ of any word in }A_j \text{ are contained in } A_i. 
\end{equation}
Then the set 
\[E=\left\{x\in [0,1): \lim_{n\to\infty}\frac{\ell_n(x,\mathbf{A})}{\phi(n)}=0\text{ and } \lim_{n\to\infty}\ell_n(x,\mathbf{A})=\infty\right\}.\]
has Hausdorff dimension at least $1-\tau$. 
\end{thm}

In the set $E$ in Theorem \ref{thm2}, once the requirement that $\lim_{n\to\infty}\ell_n(x,\mathbf{A})=\infty$ is omitted, the set always has full Hausdorff dimension whenever $\tau<1$, as is shown in Proposition \ref{prop1}.  

We prove Theorem \ref{thm1} in Section \ref{S2}. In Section \ref{S3}, we give the proofs of Theorem \ref{thm2} and Proposition \ref{prop1}. 

\section{Proof of Theorem \ref{thm1}}\label{S2}

Let $\lambda$ be the Lebesgue measure on $[0,1)$.  Let $T:[0,1)\to [0,1)$ be the doubling map $x\mapsto \{2x\}$, which maps $x\in [0,1)$ to the fractional part of $2x$. Clearly, if $x\in [0,1)$ has dyadic expansion $x=0.x_1x_2\ldots$, then the dyadic expansion of  $Tx$ is $Tx=0.x_2x_3\ldots$. Moreover, it is well-known that $\lambda$ is $T$-invariant. That is, $$\lambda(T^{-1}A)=\lambda(A)$$
 for any Borel set $A\subset [0,1)$. Also, it is easy to see that  given  $u_j\in \Sigma_{k_j}$ for $1\leq j\leq \ell$, we have  
 \begin{equation}\label{eqinde}
 \lambda(\{x\in [0,1): x_{i_j+1}\ldots x_{i_j+k_j}=u_j\text{ for }1\leq j\leq \ell\})=2^{-(k_1+\cdots+k_{\ell})}
 \end{equation}
 for any $i_j\geq0$ satisfying $i_j+k_j<i_{j+1}+1$ for $1\leq j<\ell$. These two facts will be used in the proof of Theorem \ref{thm1}. 

\begin{proof}[Proof of Theorem \ref{thm1}]
We first  prove that for $\lambda$-a.e. $x\in [0,1)$,  
\[\limsup_{n\to\infty}\frac{\ell_n(x,\mathbf{A})}{\log_2n}\leq \frac{1}{1-\tau}.\]
We assume that $\tau<1$, since otherwise the conclusion holds trivially as $\frac{1}{0}:=\infty$. 

Take $0<\epsilon<1-\tau$. For $n\in\N$, let 
\[\gamma_n(\epsilon)=\left\lceil \frac{1+\epsilon}{1-\tau-\epsilon}\log_2n\right\rceil \  \text{ and } \ E_n(\epsilon)=\{x\in[0,1): \ell_n(x,\mathbf{A})\geq \gamma_n(\epsilon)\}.\]
Here and afterwards, $\lceil \cdot\rceil$  stands for the smallest integer not less than a real number. 
For $0\leq i\leq n-\gamma_n(\epsilon)$, set 
\[F_{n,i}(\epsilon)=\left\{x\in [0,1): x_{i+1}\ldots x_{i+\gamma_n(\epsilon)}\in A_{\gamma_n(\epsilon)}\right\}.\]
Then by the definition of $\ell_n(x,\mathbf{A})$ and the assumption (i) in Theorem \ref{thm1}, we have
\[E_n(\epsilon)\subset \bigcup_{i=0}^{n-\gamma_n(\epsilon)}F_{n,i}(\epsilon).\]
Since $\lim_{n\to\infty}\frac{\log_2|A_n|}{n}=\tau$, we see that  for all sufficiently large $n$, 
$$|A_{\gamma_n(\epsilon)}|\leq 2^{\gamma_n(\epsilon)(\tau+\epsilon)}.$$ 
Hence for all large $n$, 
\[\lambda(F_{n,i}(\epsilon))=|A_{\gamma_n(\epsilon)}|2^{-\gamma_n(\epsilon)}\leq 2^{\gamma_n(\epsilon)(\tau+\epsilon-1)},\]
where in the first equality we have used the fact that $\lambda$ is $T$-invariant. 
Therefore, for all large $n$, we have 
\begin{align*}
\lambda(E_n(\epsilon))&\leq \sum_{i=0}^{n-\gamma_n(\epsilon)}\lambda(F_{n,i}(\epsilon))\\
&\leq  n\cdot 2^{\gamma_n(\epsilon)(\tau+\epsilon-1)}\\
&\leq n^{-\epsilon}.
\end{align*}

For $k\in\N$, let $m_k=\lfloor 2^{\frac{1-\tau-\epsilon}{1+\epsilon}k}\rfloor$. Then for all large $k$, 
\begin{align*}
\lambda(E_{m_k}(\epsilon))&\leq m_k^{-\epsilon}\\
&\leq \left(2^{\frac{1-\tau-\epsilon}{1+\epsilon}k}-1\right)^{-\epsilon}\\
&\leq \left(1-2^{\frac{\tau+\epsilon-1}{1+\epsilon}}\right)^{-\epsilon}2^{-\frac{(1-\tau-\epsilon)\epsilon k}{1+\epsilon}}.
\end{align*}
Thus we have 
\[\sum_{k=1}^{\infty}\lambda(E_{m_k}(\epsilon))<\infty.\]
Therefore by the Borel-Cantelli lemma, $\lambda$-a.e. $x\in [0,1)$ is contained in $E_{m_k}(\epsilon)$ for at most finitely many $k$. Note that for any $n\in\N$ with $m_{k-1}<n\leq m_k$, we have $\gamma_n(\epsilon)=\gamma_{m_k}(\epsilon)=k$. Hence $E_n(\epsilon)\subset E_{m_k}(\epsilon)$. Therefore, $\lambda$-a.e. $x\in [0,1)$ is contained in $E_{n}(\epsilon)$ for at most finitely many $n$, which implies that 
\[\limsup_{n\to\infty}\frac{\ell_n(x,\mathbf{A})}{\gamma_n(\epsilon)}\leq 1.\]
It then follows from the definition of $\gamma_n(\epsilon)$ and the arbitrariness of $\epsilon$ that for $\lambda$-a.e. $x\in [0,1)$,  
\[\limsup_{n\to\infty}\frac{\ell_n(x,\mathbf{A})}{\log_2n}\leq \frac{1}{1-\tau}.\]

Next we prove that for $\lambda$-a.e. $x\in [0,1)$, 
\[\liminf_{n\to\infty}\frac{\ell_n(x,\mathbf{A})}{\log_2n}\geq \frac{1}{1-\tau}.\]
For any $\epsilon>0$ and $n\in\N$, let $\delta_n=\left\lfloor\frac{1-\epsilon}{1-\tau+\epsilon}\log_2n\right\rfloor$. Set 
\[G_n(\epsilon)=\left\{x\in [0,1): x_{i\delta_n^2+1}\ldots x_{i\delta_n^2+\delta_n}\not\in A_{\delta_n}\text{ for all }0\leq i<\left\lfloor \frac{n}{\delta_n^2}\right\rfloor\right\}.\]
Then notice that
\begin{equation}\label{eqellnxAl}
\{x\in [0,1): \ell_n(x,\mathbf{A})<\delta_n\}\subset G_n(\epsilon).
\end{equation}
Since $\lim_{n\to\infty}\frac{\log_2|A_n|}{n}=\tau$, we have $|A_{\delta_n}|\geq 2^{(\tau-\epsilon)\delta_n}$ for all large $n$. Hence by the fact \eqref{eqinde}, we see that  for all large $n$,  
\begin{align*}
\lambda(G_n(\epsilon))&=2^{-\left(\delta_n\cdot\left\lfloor\frac{n}{\delta_n^2}\right\rfloor\right)}\left(2^{\delta_n}-|A_{\delta_n}|\right)^{\left\lfloor \frac{n}{\delta_n^2}\right\rfloor}\\
&=\left(1-2^{-\delta_n}|A_{\delta_n}|\right)^{\left\lfloor\frac{n}{\delta_n^2}\right\rfloor}\\
&\leq \left(1-2^{-\delta_n}2^{(\tau-\epsilon)\delta_n}\right)^{\frac{n}{2\delta_n^2}}\\
&\leq e^{-(2^{(\tau-\epsilon-1)\delta_n})\cdot\frac{n}{2\delta_n^2}}\\
&\leq e^{-\frac{n^{\epsilon}}{2\delta_n^2}}.
\end{align*}
Hence it follows from the Borel-Cantelli lemma and \eqref{eqellnxAl} that for $\lambda$-a.e. $x\in [0,1)$, $\ell_n(x,\mathbf{A})\geq \delta_n$ for all  large $n$. Therefore, by the definition of $\delta_n$ we see that for $\lambda$-a.e. $x\in [0,1)$,
\begin{equation}\label{eqntoinqtau}
\liminf_{n\to\infty}\frac{\ell_n(x,\mathbf{A})}{\log_2n}\geq \frac{1-\epsilon}{1-\tau+\epsilon}.
\end{equation}
Since $\epsilon>0$ is arbitrary, we thus have  for $\lambda$-a.e. $x\in [0,1)$,
\[\liminf_{n\to\infty}\frac{\ell_n(x,\mathbf{A})}{\log_2n}\geq \frac{1}{1-\tau}.\] 
We note that  the case that $\tau=1$ follows from \eqref{eqntoinqtau} and the convention that $\frac{1}{0}:=\infty$. 
This completes the proof of the theorem. 
\end{proof}

\section{Proof of Theorem \ref{thm2}}\label{S3}

Given $k,\ell\in\N$ and a word $u\in\Sigma_{\ell}$, define  
\[u\Sigma_k:=\{uw: w\in \Sigma_k\},\]
which is a  subset  of $\Sigma_{k+\ell}$. We first give an elementary lemma which is important in the proof of Theorem \ref{thm2}. 

\begin{lem}\label{lemew}
Let $\mathbf{A}=\{A_k\}_{k=1}^{\infty}$ be a sequence of sets where  each  $A_k$ is a non-empty subset of $\Sigma_k$. Let $\tau=\limsup_{k\to\infty}\frac{\log_2|A_k|}{k}$. Suppose that $\tau<1$ and
let $s>0$ such that $(1+s)\tau<1$. Then for all  large enough $n$, there exists $u_{n}\in \Sigma_n$ such that 
\[u_{n}\Sigma_{\lfloor sn\rfloor}\cap A_{n+\lfloor sn\rfloor}=\emptyset.\]
\end{lem}

\begin{proof}
Let $\epsilon>0$ such that $(1+s)(\tau+\epsilon)<1$. Since $\tau=\limsup_{n\to\infty}\frac{\log_2|A_n|}{n}$,  we see that for all sufficiently large $n$, 
\begin{equation}\label{eqAnsnr}
|A_{n+\lfloor sn\rfloor}|\leq 2^{(n+\lfloor sn\rfloor)(\tau+\epsilon)}\leq 2^{n(1+s)(\tau+\epsilon)}<2^n.
\end{equation}
If for every $u\in\Sigma_n$, one has $u\Sigma_{\lfloor sn\rfloor}\cap A_{n+\lfloor sn\rfloor}\neq\emptyset$, then clearly $$|A_{n+\lfloor sn\rfloor}|\geq |\Sigma_n|=2^n,$$
which contradicts \eqref{eqAnsnr}. Hence there exists $u_{n}\in\Sigma_n$ such that $u_{n}\Sigma_{\lfloor sn\rfloor}\cap A_{n+\lfloor sn\rfloor}=\emptyset$, completing the proof of the lemma. 
\end{proof}

In our proof of Theorem \ref{thm2}, we will make use of a dimensional result on  the {\it homogeneous Moran sets} established in \cite{FWW97S}. For $F\subset \R^d$, let $\dim_{\rm H}F$ denote the Hausdorff dimension of $F$. 

Let us first recall the definition of a homogeneous Moran set. Let $\{n_k\}_{k=1}^{\infty}$ be a sequence of positive integers and $\{c_k\}_{k=1}^{\infty}$ be a sequence of positive numbers satisfying 
\begin{equation}\label{morancon}
n_k\geq 2, \quad 0<c_k<1, \quad  \text{ and } n_kc_k\leq 1 \quad \text{ for all }k\geq 1.
\end{equation}
For any $k\geq 1$, let $$D_k=\{(i_1,\ldots,i_k):1\leq i_j\leq n_j, 1\leq j\leq k\}.$$ Let $D=\bigcup_{k\geq 0}D_k$, where $D_0=\emptyset$. For $\sigma=(\sigma_1,\ldots,\sigma_k)\in D_k$ and $\tau=(\tau_1,\ldots,\tau_{\ell})\in D_{\ell}$, let $\sigma\tau=(\sigma_1,\ldots,\sigma_k,\tau_1,\ldots,\tau_{\ell})$ be the concatenation of $\sigma$ and $\tau$. 

Let $J$ be the interval $[0,1]$. Suppose $\mathcal{F}=\{J_{\sigma}:\sigma\in D\}$ is a collection of closed subintervals of $J$ satisfying 
\begin{itemize}
\item[(i)]$J_{\emptyset}=J$;
\item[(ii)] For any $k\geq 1$ and $\sigma\in D_{k-1}$, $J_{\sigma1},\ldots, J_{\sigma n_k}$ are subintervals of $J_{\sigma}$ and ${\rm int}(J_{\sigma i})\cap {\rm int} (J_{\sigma j})=\emptyset$ for any $i\neq j$, where ${\rm int}(F)$ denotes the interior of a set  $F$;
\item[(iii)] For any $k\geq 1$, $\sigma\in D_{k-1}$ and  $1\leq i\leq n_k$, we have 
\[c_k=\frac{\lambda(J_{\sigma i})}{\lambda(J_{\sigma})}.\]
\end{itemize} 
Then we call the set 
\begin{equation}\label{eqEmoran}
E=\bigcap_{k\geq 1}\bigcup_{\sigma\in D_k}J_{\sigma}
\end{equation}
the {\it homogeneous Moran set} generated by $\mathcal{F}$. 

\begin{lem}\cite[Theorem 2.1]{FWW97S}\label{lemmoran}
Let $E$ be the homogeneous Moran set defined as above. Then we have 
\[\dim_{\rm H}E\geq \liminf_{k\to\infty}\frac{\log n_1n_2\cdots n_k}{-\log c_1c_2\cdots c_{k+1}n_{k+1}}.\] 
\end{lem}

\begin{rem}\label{remMoran}
From the proof of \cite[Theorem 2.1]{FWW97S} it is easily seen that the conclusion of Lemma \ref{lemmoran} still holds, if the assumption   \eqref{morancon} is replaced by the following:
\[n_k\geq 1, \quad \sup_{k\geq 1}c_k<1, \quad \text{ and } n_kc_k\leq 1 \quad \text{ for all }k\geq 1.\]
In this case, we still call the set $E$ defined in \eqref{eqEmoran} a homogeneous Moran set. 
\end{rem}

Now we are ready to prove Theorem \ref{thm2}. 

\begin{proof}[Proof of Theorem \ref{thm2}]
When $\tau=1$, clearly there is  nothing to prove. So we assume that $\tau<1$. Let $s>0$ such that 
\begin{equation*}\label{eq1pluss}
 (1+s)\tau<1.
\end{equation*}
For  $k\in\N$, let $m_k=\lfloor \sqrt{\phi(k)}\rfloor$. Then by Lemma \ref{lemew} and the assumption that $\lim_{n\to\infty}\phi(n)=\infty$, we can assume that for each $k\in\N$, there exists $u_{m_k}\in \Sigma_{m_k}$ such that 
\begin{equation}\label{equmks}
u_{m_k}\Sigma_{\lfloor sm_k\rfloor}\cap A_{m_k+\lfloor sm_k\rfloor}=\emptyset. 
\end{equation}

Fix a sequence of words $\{\xi_{m_k}\}_{k=1}^{\infty}$, where $\xi_{m_k}\in A_{m_k}$ ($k\geq 1$).  For $k\in\N$, define 
\[W_k=\{u_{m_k}v_1u_{m_k}v_2\ldots u_{m_k}v_{m_k}\xi_{m_k}: v_i\in \Sigma_{\lfloor sm_k\rfloor}, 1\leq i\leq m_k\}.\]
Then  $W_k$ is a collection of words each of which  has length
\begin{equation}\label{eqlength}
\ell_k=(m_k+\lfloor sm_k\rfloor)m_k+m_k.
\end{equation} 
Let 
\[W=\{0.w_1w_2\ldots\in [0,1]: w_k\in W_k, k\in\N\}.\]
In the following, we show that $W\subset E$  and has Hausdorff dimension at least $\frac{s}{s+1}$.

Let $x\in W$. For $n\in\N$, let $k\in\N$ be such that 
\begin{equation}\label{eql1toln}
\ell_1+\cdots+\ell_{k-1}<n\leq \ell_1+\cdots+\ell_k.
\end{equation}
Then by \eqref{equmks}, the definitions of $W_k$ ($k\geq 1$) and $W$, and the assumption \eqref{assumption} in Theorem \ref{thm2}, we see that 
\[m_{k-1}\leq \ell_n(x,\mathbf{A})\leq 3(m_k+\lfloor sm_k\rfloor).\]
From \eqref{eql1toln} we infer that $k\leq n$. Hence by the monotonicity of $\{\phi(n)\}_{n=1}^{\infty}$ we have 
\[\ell_n(x,\mathbf{A})\leq 3(1+s)m_k\leq 3(1+s)m_n=3(1+s)\lfloor \sqrt{\phi(n)}\rfloor.\]
Therefore,
\[\lim_{n\to\infty}\frac{\ell_n(x,\mathbf{A})}{\phi(n)}=0 \quad \text{ and } \quad \lim_{n\to\infty}\ell_n(x,\mathbf{A})=\infty.\] Since $x\in W$ is arbitrary, it follows that  $W\subset E$.

Let $c_k=\frac{1}{2}$ for all $k\geq 1$ and let  $\{n_k\}_{k=1}^{\infty}$ be a sequence of positive integers defined as follows.  For $k\geq 1$,  let $t\geq 1$ be such that 
\[\ell_1+\cdots+\ell_{t-1}<k\leq \ell_1+\cdots+\ell_t.\]
(Set $\ell_0=0$.) Then define 
\begin{equation*}
n_k=\begin{cases} 1 & \ \text{ if }k\in \mathcal{I}_t,\\
2& \ \text{ if } k\in \mathcal{J}_t,
\end{cases}
\end{equation*}
where
\[\mathcal{I}_t=\{\ell_1+\cdots+\ell_{t-1}+i(m_t+\lfloor sm_t\rfloor)+j: 0\leq i\leq m_t, 1\leq j\leq m_t\}\]
and 
\[\mathcal{J}_t=\{\ell_1+\cdots+\ell_{t-1}+i(m_t+\lfloor sm_t\rfloor)+m_t+j: 0\leq i<m_t, 1\leq j\leq \lfloor sm_t\rfloor\}.\]
In other words, $n_k$ is defined to be $1$ at the positions determined by the sequences $\{u_{m_i}\}_{i=1}^{\infty}$ and $\{\xi_{m_i}\}_{i=1}^{\infty}$ in the digits of numbers in $W$, and $2$ at other positions.  
Then by considering dyadic intervals, it is not difficult to see that  $W$ is a  homogeneous Moran set with parameters $\{n_k\}_{k=1}^{\infty}$ and $\{c_k\}_{k=1}^{\infty}$ (in the sense in Remark \ref{remMoran}). 

Next we show that $\dim_{\rm H}W\geq \frac{s}{s+1}$. For $k\in\N$, let $j\geq 1$ be such that 
\[\ell_1+\cdots+\ell_{j-1}<k\leq \ell_1+\cdots+\ell_{j}.\]
Then we have 
\[n_1n_2\cdots n_k\geq 2^{|\mathcal{J}_1|}\cdot 2^{|\mathcal{J}_2|}\cdots2^{|\mathcal{J}_{j-1}|}=2^{m_1\lfloor sm_1\rfloor+m_2\lfloor sm_2\rfloor +\cdots +m_{j-1}\lfloor sm_{j-1}\rfloor}\]
and 
\[c_1c_2\cdots c_{k+1}n_{k+1}\geq 2^{-(k+1)}\geq 2^{-1}2^{-(\ell_1+\cdots+\ell_j)}.\]
Therefore, by Lemma \ref{lemmoran} and Remark \ref{remMoran},
\begin{align*}
\dim_{\rm H}W&\geq \liminf_{k\to\infty}\frac{\log n_1n_2\cdots n_k}{-\log c_1c_2\cdots c_{k+1}n_{k+1}}\\
&\geq \liminf_{j\to\infty}\frac{m_1\lfloor sm_1\rfloor+\cdots+m_{j-1}\lfloor sm_{j-1}\rfloor}{\ell_1+\cdots+\ell_j}\\
&= \frac{s}{s+1},
\end{align*}
where the last equality follows from the definition of $m_k$ ($k\geq 1$), \eqref{eqlength} and the assumption that $\lim_{n\to\infty}\frac{\phi(n)}{\phi(1)+\cdots+\phi(n-1)}=0$. Therefore, we have  
\begin{equation}\label{eqdimEs}
\dim_{\rm H}E\geq \frac{s}{s+1}.
\end{equation}

When $\tau=0$,  since $s$ satisfying $(1+s)\tau<1$ can be arbitrarily large and \eqref{eqdimEs} holds for any such $s$, hence in this case we have $$\dim_{\rm H}E=1=1-\tau.$$  When $0<\tau<1$, since \eqref{eqdimEs} holds for any $s$ satisfying $(1+s)\tau<1$ (which is equivalent to that $s<\frac{1}{\tau}-1$), we thus see that in this case,  $$\dim_{\rm H}E\geq \frac{\frac{1}{\tau}-1}{(\frac{1}{\tau}-1)+1}=1-\tau.$$
This completes the proof of the theorem.  
\end{proof}

In the end of this paper, we prove the following result which complements Theorem \ref{thm2}. 

\begin{prop}\label{prop1}
Let $\mathbf{A}$ be as in Theorem \ref{thm2}. Suppose that $\tau<1$. Then 
\[\dim_{\rm H}\left\{x\in [0,1): \lim_{n\to\infty}\ell_n(x,\mathbf{A})<\infty\right\}=1.\]
\end{prop}

\begin{proof}
Since $\limsup_{k\to\infty}(\log_2|A_k|)/k=\tau<1$, we see that for all sufficiently large $N$, $\Sigma_N\setminus A_N$ has cardinality not less than $2^{N-1}$. Fix such an $N$.  Observe that by the assumption \eqref{assumption} in Theorem \ref{thm2}, we have 
\begin{align*}
F&:=\{0.w_1w_2\ldots: w_i\in \Sigma_N\setminus A_N,i\geq1\}\\
&\subset \{x\in[0,1): \ell_n(x,\mathbf{A})<2N\text{ for all }n\geq1\}\\
&\subset \left\{x\in [0,1): \lim_{n\to\infty}\ell_n(x,\mathbf{A})<\infty\right\}.
\end{align*}
On the other hand, we note that $F$ is a self-similar set generated by the following iterated function system which satisfies the open set condition:
\[\Phi=\left\{2^{-N}x+a_w: w\in \Sigma_N\setminus A_N\right\},\]
where for $w=u_1\ldots u_N\in\Sigma_N\setminus A_N$, $a_w=\sum_{i=1}^N\frac{u_i}{2^i}$. Since $\Phi$ satisfies the open set condition, consists of at least $2^{N-1}$  similitudes with ratio $2^{-N}$, we know from the basic dimension theory of self-similar sets (see e.g. \cite{Falconer14}) that $\dim_{\rm H}F\geq \frac{N-1}{N}$.  Therefore, 
\[\dim_{\rm H}\left\{x\in [0,1): \lim_{n\to\infty}\ell_n(x,\mathbf{A})<\infty\right\}\geq \frac{N-1}{N}.\]
Letting $N\to\infty$, we complete the proof of the proposition. 
\end{proof}


\end{document}